\documentclass[12pt]{amsart}
\usepackage[utf8]{inputenc}

\usepackage{amsrefs}
\usepackage{supertabular}
\usepackage{amsxtra,amssymb,amsmath,amscd,url,enumitem}

\setenumerate{itemsep=5pt}
\setitemize{itemsep=5pt}

\newlist{enumth}{enumerate}{1}
\setlist[enumth]{label=\emph{(\arabic*)}, ref=\thetheorem(\arabic*)}

\usepackage[utf8]{inputenc}
\usepackage[english]{babel}
\usepackage{fullpage}
\usepackage{scrtime}
\usepackage[colorlinks]{hyperref}
\usepackage{datetime}
\usepackage[all]{xy}
\usepackage{tikz-cd}
\usepackage{mathrsfs}
\usepackage{xspace}

\usepackage{tensor}
\usepackage{braket}

\usepackage{todonotes}
\usepackage[capitalize]{cleveref}
\crefalias{enumthi}{theorem}
\crefname{section}{\S}{\SS}
\crefname{figure}{Figure}{Figures}

\shortdate

\DeclareMathSymbol{A}{\mathalpha}{operators}{`A}%
\DeclareMathSymbol{B}{\mathalpha}{operators}{`B}%
\DeclareMathSymbol{C}{\mathalpha}{operators}{`C}%
\DeclareMathSymbol{D}{\mathalpha}{operators}{`D}%
\DeclareMathSymbol{E}{\mathalpha}{operators}{`E}%
\DeclareMathSymbol{F}{\mathalpha}{operators}{`F}%
\DeclareMathSymbol{G}{\mathalpha}{operators}{`G}%
\DeclareMathSymbol{H}{\mathalpha}{operators}{`H}%
\DeclareMathSymbol{I}{\mathalpha}{operators}{`I}%
\DeclareMathSymbol{J}{\mathalpha}{operators}{`J}%
\DeclareMathSymbol{K}{\mathalpha}{operators}{`K}%
\DeclareMathSymbol{L}{\mathalpha}{operators}{`L}%
\DeclareMathSymbol{M}{\mathalpha}{operators}{`M}%
\DeclareMathSymbol{N}{\mathalpha}{operators}{`N}%
\DeclareMathSymbol{O}{\mathalpha}{operators}{`O}%
\DeclareMathSymbol{P}{\mathalpha}{operators}{`P}%
\DeclareMathSymbol{Q}{\mathalpha}{operators}{`Q}%
\DeclareMathSymbol{R}{\mathalpha}{operators}{`R}%
\DeclareMathSymbol{S}{\mathalpha}{operators}{`S}%
\DeclareMathSymbol{T}{\mathalpha}{operators}{`T}%
\DeclareMathSymbol{U}{\mathalpha}{operators}{`U}%
\DeclareMathSymbol{V}{\mathalpha}{operators}{`V}%
\DeclareMathSymbol{W}{\mathalpha}{operators}{`W}%
\DeclareMathSymbol{X}{\mathalpha}{operators}{`X}%
\DeclareMathSymbol{Y}{\mathalpha}{operators}{`Y}%
\DeclareMathSymbol{Z}{\mathalpha}{operators}{`Z}%

\renewcommand{\leq}{\leqslant}
\renewcommand{\geq}{\geqslant}
 
\numberwithin{equation}{section}

\renewcommand{\mathcal}{\mathscr}

\newcommand{\Cc}{\mathbf{C}}

\newcommand{\Zz}{\mathbf{Z}}

\newcommand{\Rr}{\mathbf{R}}

\newcommand{\Qq}{\mathbf{Q}}

\newcommand{\Ff}{\mathbf{F}}

\newcommand{\expect}{\mathbf{E}}

\def\loccit{loc.\kern3pt cit.{}\xspace}
\def\cf{see\kern.3em}
\def\Cf{See\kern.3em}
\def\eg{e.g.\kern.3em}

\def\resp{\text{resp.}\kern.3em}

\newcommand{\mods}[1]{\,(\mathrm{mod}\,{#1})}

\DeclareMathOperator{\Tr}{Tr}

\newcommand{\eps}{\varepsilon}
\renewcommand{\rho}{\varrho}

\DeclareMathOperator{\SL}{\mathbf{SL}}
\DeclareMathOperator{\GL}{\mathbf{GL}}

\DeclareMathOperator{\PGL}{\mathbf{PGL}}

\DeclareMathOperator{\SO}{\mathbf{SO}}

\DeclareMathOperator{\SU}{\mathbf{SU}}
\DeclareMathOperator{\Un}{\mathbf{U}}

\newcommand{\demi}{{\textstyle{\frac{1}{2}}}}

\DeclareMathSymbol{\gena}{\mathord}{letters}{"3C}
\DeclareMathSymbol{\genb}{\mathord}{letters}{"3E}

\theoremstyle{plain}
\newtheorem{theorem}{Theorem}[section]
\newtheorem*{theorem*}{Theorem}
\newtheorem{lemma}[theorem]{Lemma}

\newtheorem{proposition}[theorem]{Proposition}

\theoremstyle{remark}

\theoremstyle{definition}

\newtheorem{remark}[theorem]{Remark}

\AddToHook{env/lemma/begin}{\crefalias{theorem}{lemma}}
\AddToHook{env/proposition/begin}{\crefalias{theorem}{proposition}}
\AddToHook{env/corollary/begin}{\crefalias{theorem}{corollary}}
\AddToHook{env/conjecture/begin}{\crefalias{theorem}{conjecture}}
\AddToHook{env/claim/begin}{\crefalias{theorem}{claim}}
\AddToHook{env/remark/begin}{\crefalias{theorem}{remark}}

\newcommand{\abs}[1]{\left\lvert#1\right\rvert}

\newcommand{\mcF}{\mathscr{F}}

\renewcommand{\geq}{\geqslant}
\renewcommand{\leq}{\leqslant}

\begin{document}

\title{Spectrally indistinguishable pseudorandom graphs}

\author{Arthur Forey}
\address[A. Forey]{Univ. Lille, CNRS, UMR 8524 - Laboratoire Paul Painlevé,  \newline F-59000 Lille, France} 
  \email{arthur.forey@univ-lille.fr}

\author{Javier Fresán}
\address[J. Fres\'an]{Sorbonne Université and Université Paris Cité, CNRS, IMJ-PRG, \newline F-75005 Paris, France}
\email{javier.fresan@imj-prg.fr}

\author{Emmanuel Kowalski}
\address[E. Kowalski]{D-MATH, ETH Zürich, Rämistrasse 101, 8092 Zürich, Switzerland} 
\email{kowalski@math.ethz.ch}

\author{Yuval Wigderson}
\address[Y. Wigderson]{ETH-ITS, Scheuchzerstrasse 70, 8006 Zürich, Switzerland} 
\email{yuval.wigderson@eth-its.ethz.ch}


\begin{abstract}
  We construct explicit families of graphs whose eigenvalues are
  asymptotically distributed according to Wigner's semicircle law; in
  other words, that are spectrally indistinguishable from random
  graphs. However, in other respects they are strikingly dissimilar from
  random graphs; for example, they are $K_{2,3}$-free graphs with almost
  the maximum possible edge density.
\end{abstract}

\maketitle 

\section{Introduction}

One of the most important developments in the last half-century of
combinatorics and many related fields has been the notion of
\emph{pseudorandomness}. Loosely, one says that a discrete object
(say, a graph) is pseudorandom if it satisfies some
\emph{deterministic} property that is also shared, with high
probability, by a randomly chosen object.

In graph theory, the best-known and most studied notion\footnote{For
  dense graphs, the seminal works of Thomason~\cite{thomason} and
  Chung--Graham--Wilson~\cite{cgw} imply that essentially all
  notions of pseudorandomness are roughly equivalent. However, this is
  known to be false for sparse graphs (see, e.g., the paper~\cite{sstz}
  of Sah, Sawhney, Tidor and Zhao and the discussion therein).} of
pseudorandomness is often called \emph{spectral
  pseudorandomness}. Loosely speaking, one says that a graph is
spectrally pseudorandom if all of its non-trivial
eigenvalues\footnote{As usual, when referring to the eigenvalues of a
  graph, we mean the eigenvalues of its adjacency matrix.} are much
smaller in absolute value than its average degree. The 
\emph{expander mixing lemma} of Alon and Chung \cite{ac} states that such a condition
implies a fairly uniform distribution of edges among all large sets in
the graph, as one would expect in a random graph.  Of particular
importance are \emph{(near-)Ramanujan graphs}, where all the non-trivial
eigenvalues are at most roughly the square root of the degree, as these
give essentially optimal control on the distribution of edges.  For an
in-depth introduction to spectrally pseudorandom graphs, we refer the reader to the
excellent survey~\cite{ks} of Krivelevich and Sudakov.

Our focus in this paper is on the construction of explicit graphs that
are ``even more spectrally pseudorandom''. More precisely, Wigner's semicircle law~\cite{wigner} implies that, upon appropriate
scaling, the spectrum of an Erd\H os--Rényi random graph at any\footnote{To
  be exact, one must impose the weak (and necessary) assumption that
  $pn$ and $(1-p)n$ both tend to infinity, {where $n$ is the number of vertices and $p=p(n)$ the edge probability.}} edge density converges to
the semicircle distribution (see, e.g., the paper~\cite{tvw}*{\S1.2} of
Tran, Vu and Wang for the precise statement). Moreover, as is common in
probability, Wigner's result exhibits universality, and it is now known
that many other models of random graphs, such as random regular
graphs and power-law graphs, also exhibit a semicircular spectrum; see the \hbox{papers~\cites{clv,tvw}} of Chung--Lu--Vu and Tran--Vu--Wang for details.

By contrast, the explicit pseudorandom graphs that we are aware of (such
as those in~\hbox{\cite{ks}*{\S3})} have a spectral distribution that is
very far from semicircular. In fact, these constructions have a very
``spiky'' spectrum, where a small number of eigenvalues appear with
extremely large multiplicity. As such, the spectral distribution does
not converge to \emph{any} absolutely continuous distribution, but
rather to a finitely-supported atomic measure. For example, Paley graphs
and many graphs arising from finite geometries are \emph{strongly
  regular}, implying that they only have two non-trivial eigenvalues,
both of which occur with multiplicity linear in the number of vertices.

Our main result in this paper is an explicit construction of graphs that turn out not only to be spectrally pseudorandom in the strongest
possible sense (they are nearly Ramanujan graphs), but also to have the property that their empirical spectral distribution
converges to the semicircle distribution. However, in other respects,
they are highly unlike random graphs, as exhibited by the avoidance of
certain small subgraphs.

\begin{theorem}\label{th-main}
  Let~$k$ be a finite field. Set\footnote{\ The letters $K$ and $B$
    stand for ``Kloosterman'' and ``Birch'' respectively, for reasons
    that will be clear in the course of the proof in
    Section~\ref{sec:proof}.}
  \[
    K(k)=\{(x,y)\in k\times k\,\mid\, xy=1\} \quad\text{and}\quad
    B(k)=\{(x,y)\in k\times k\,\mid\, y=x^3\}.
  \]

  Define graphs $\Gamma_K(k)$ and~$\Gamma_B(k)$ with vertex set
  $k\times k$ in both cases and with edges joining~$x$ and~$y$ if and
  only if $x+y\in K(k)$ or $x+y\in B(k)$, respectively. These are
  regular graphs of degree~$|k|-1$ and $|k|$, respectively, and satisfy
  the following additional properties:

  \begin{enumth}
  \item The complete bipartite graph $K_{2,3}=\tikz[every node/.style={fill, circle, inner sep=1.2pt},xscale=.35, yscale=.25,baseline=2.1ex]{\foreach \i in {1,2} \node (a\i) at (0,\i+.5) {}; \foreach \j in {1,2,3} \node (b\j) at (1,\j) {}; \foreach \i in {1,2} {\foreach \j in {1,2,3} \draw (a\i) -- (b\j);}}$ is not a
    subgraph of $\Gamma_K(k)$.

  \item All the non-trivial eigenvalues~$\lambda$ of the adjacency
    matrix of~$\Gamma_K(k)$ satisfy
    \[
      |\lambda|\leq 2|k|^{1/2}.
    \]
    
  \item As the size of~$k$ tends to infinity, the numbers
    \[
      \Bigl\{\frac{\lambda}{\sqrt{|K(k)|}}\,\mid\, \lambda\text{ an
        eigenvalue of }\Gamma_K(k)\Bigr\}
    \]
    converge in distribution to the semicircle distribution
    \begin{equation}\label{eq-musc}
      \mu_{\mathrm{sc}} = \frac{1}{\pi} \sqrt{1- \frac{x^2}{4}}dx
    \end{equation}
    on $[-2,2]$. In other words, for all $-2 \leq a <b \leq 2$, we have
    \[
      \frac{1}{|k|^2} \Bigl|\Bigl\{\lambda \text{ an eigenvalue of
      }\Gamma_K(k) \,\mid\, \frac{\lambda}{\sqrt{|K(k)|}} \in
      [a,b]\Bigr\}\Bigr| \to \frac{1}{\pi} \int_a^b  \sqrt{1-
        \frac{x^2}{4}}dx
    \]
    as $|k|\to+\infty$. In fact, we have the estimate
    \begin{equation}\label{eq-w1}
      W_1\Bigl( \frac{1}{|k|^2-1}
      \sum_{\lambda\not=|k|-1}\delta_{\lambda/\sqrt{|K(k)|}},
      \mu_{\mathrm{sc}}\Bigr)
      =O(|k|^{-1/3}),
    \end{equation}
    where~$\delta_t$ denotes a Dirac mass at~$t$ and~$W_1$ denotes the
    Wasserstein distance for probability measures and the sum is over
    the non-trivial eigenvalues of the graph~$\Gamma_K(k)$.
  \item The same properties hold for the graphs~$\Gamma_B(k)$,
    provided~$k$ has characteristic~$\geq 7$.
  \end{enumth}
\end{theorem}

\begin{figure}[ht]
  \includegraphics[width=\textwidth]{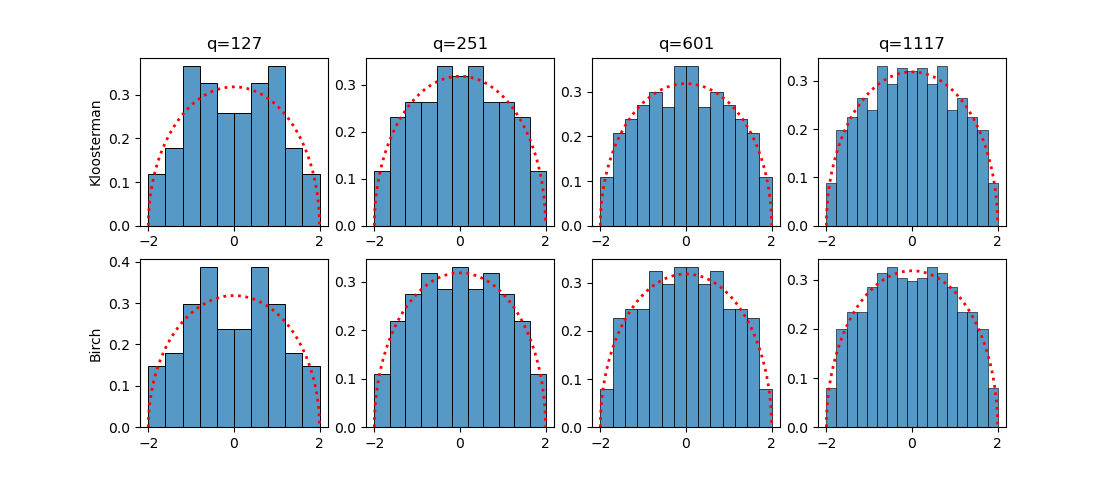}
  \caption{The normalized spectrum of the graphs $\Gamma_K(k)$ (first row) and $\Gamma_B(k)$ (second row), for $k = \Ff_{127}, \Ff_{251},\Ff_{601},\Ff_{1117}$. The dashed red curve is the density function of $\mu_{\mathrm{sc}}$, the semicircle distribution.}
\end{figure}

\begin{remark}
  We now make a few remarks about the statement of \cref{th-main}.
  \begin{enumerate}
  \item As is often the case in explicit constructions of pseudorandom
    graphs, both~$\Gamma_K(k)$ and $\Gamma_B(k)$ may have up to $d$
    vertices with loops (where $d$ is the degree). One can delete these
    loops and obtain graphs with essentially the same properties, but as
    is standard, it is more convenient for the analysis to include them.
  \item The bound $\abs \lambda \leq 2\sqrt {\abs k} = (2+o(1))\sqrt{d}$
    means that these graphs are near-Ramanujan graphs, and hence they
    enjoy essentially optimal pseudorandomness in the traditional sense;
    see, e.g., the survey of Krivelevich and Sudakov~\cite{ks}*{p.\,19}.
  \item The fact that these graphs do not contain a copy of $K_{2,3}$ is
    quite surprising. Indeed, the Kővári--Sós--Turán theorem~\cite{kst}
    implies that \emph{every} $n$-vertex graph with average degree at
    least $(\sqrt 2+o(1))\sqrt n$ contains a copy of $K_{2,3}$, and it
    was shown by~Füredi~\cite{furedi} that this bound is asymptotically
    sharp. That is, up to the constant factor $\sqrt{2}$, the graphs
    $\Gamma_K(k)$ and~$\Gamma_B(k)$ are as dense as possible among all
    $K_{2,3}$-free graphs, and in this sense extremely atypical among
    all graphs of the same edge density.

    In fact, an Erdős--Rényi random graph at the same density contains
    a copy of $K_{2,3}$ asymptotically almost surely (in fact, it has on
    the order of $n^2$ such copies). The same holds in other natural
    random graph models at the same density, such as random regular
    graphs with degree $d = \sqrt n$.
  \item The restriction to characteristic $\geq 7$ is necessary for the
    statement concerning the spectrum of the graphs
    $\Gamma_B(k)$. Indeed, if the characteristic of~$k$ is $2$ or $5$,
    then the limiting spectrum of $\Gamma_B(k)$ is a
    finitely\nobreakdash-supported atomic measure
    (see Remark~\ref{rm-finite} for details). On the other hand, these
    graphs are still $K_{2,3}$-free.

    If the characteristic of~$k$ is $3$, then one can check that
    $\Gamma_B(k)$ is a disjoint union of~$(\abs k-1)/2$ copies of the
    complete bipartite graph $K_{\abs k, \abs k}$, as well as one copy
    of the complete graph $K_{\abs k}$, and hence none of the properties
    above hold in this case.
    
  \item In \cref{th-main}, we define the graphs $\Gamma_B(k)$ and
    $\Gamma_K(k)$ to be Cayley sum graphs. However, every statement in
    \cref{th-main} would remain true if, instead, we defined them as
    Cayley graphs of the same groups with the same generating
    sets. However, for the remainder of the paper we work with Cayley
    sum graphs, because some of the variant constructions we consider
    later in Section \ref{sec:subsets} are defined in terms of
    non-symmetric generating sets, and hence do not naturally have
    (undirected) Cayley graphs. By working with Cayley sum graphs
    throughout, we can avoid having to worry about the symmetry of the
    generating set.

    \item The Wasserstein (also called Monge--Kontorovich or
      Rubinstein--Kontorovich) distance $W_1(\mu_1,\mu_2)$ between
      probability measures $\mu_1, \mu_2$ on a compact metric space~$X$ (in our case
      the interval $[-2,2]$) can be defined as
      \[
        W_1(\mu_1,\mu_2)=\sup_{f\text{ $1$-Lipschitz}}\Bigl|
        \int_{X}fd\mu_1-\int_Xfd\mu_2\Bigr|,
      \]
      where~$f$ runs over $1$-Lipschitz functions $X\to \Cc$ (see~\cite{k-u}
      for an introduction to Wasserstein metrics from the point of view of
      equidistribution).

    \item In a companion paper~\cite{jacobian}, we will describe other
      classes of examples of graphs with properties similar to those
      of $\Gamma_K(k)$ and~$\Gamma_B(k)$.  These are constructed using
      jacobians (and generalized jacobians) of algebraic curves over
      finite fields, and will show that there are very rich families
      of high-density $K_{2,3}$-free graphs with semicircular spectral
      distribution. The proofs of these will however require more
      sophisticated results in algebraic and arithmetic geometry.
  \end{enumerate}
\end{remark}

The remainder of this paper is organized as follows. In
Section \ref{sec:discussion}, we discuss our motivation for proving
\cref{th-main}, and why we consider it interesting and surprising. In
Section \ref{sec:proof}, we prove \cref{th-main}; our proof is quite short,
since the $K_{2,3}$-freeness of the graphs $\Gamma_K(k)$ and~$\Gamma_B(k)$ is quite elementary to prove, and the remaining properties
follow rather quickly from known, but very deep, results in number
theory and algebraic geometry. In Section \ref{sec:larsen}, we give an
alternative, and more self-contained, proof that the spectrum of
$\Gamma_K(k)$ is given by the semicircular distribution. Notably, this
proof exhibits a surprising connection between the spectral distribution
and the $K_{2,3}$-freeness of $\Gamma_K(k)$; this $K_{2,3}$-freeness is
critical to allow us to apply a group-theoretic tool called
\emph{Larsen's alternative}, which is itself the main tool towards
identifying the semicircle distribution as the limiting spectral
distribution. Finally, in Section \ref{sec:subsets}, we discuss a few variants
of our construction; in particular, we can similarly construct
$C_4$-free graphs of nearly optimal density, whose limiting spectral
distributions are no longer semicircular, but are specific measures that
we can describe and analyze.

\subsection*{Acknowledgments.} We thank H. T.\,Pham for pointing out
to us the paper of Soloveychik, Xiang and Tarokh, and T.\,Tokieda for
his help in finding the paper of Sunada.  We also thank T.\,Schramm
for insights on the relationship between semicircular spectra, small
subgraph counts and computational indistinguishability.

This project has received funding from the European Research Council (ERC) under the European Union’s Horizon 2020 research and innovation programme (grant agreement no.\,101170066). E.\,K.\, is partially supported by the SNF grant $219220$ and the SNF\nobreakdash-ANR ``Etiene'' grant $10003145$. Y.\,W.\,is supported by Dr.\,Max R\"ossler, the Walter Haefner Foundation
and the ETH Z\"urich Foundation.  

\section{Discussion and historical background}\label{sec:discussion}
\subsection{Why should we care?}
Why should one care about explicit graphs with semicircular spectrum? From the perspective of extremal graph theory, it is not clear that such graphs provide much advantage over other known constructions of spectrally pseudorandom graphs. Indeed, we are aware of very few applications where the knowledge of the full spectrum is meaningfully more powerful than simply bounding the largest and smallest non-trivial eigenvalues (although we return to this point in \cref{subsec:independence}).

One reason why we should care is purely philosophical: since almost all graphs (at any edge density) have a semicircular spectrum, it is very natural to look for explicit examples with the same property. This is very similar to the search for explicit normal numbers, for explicit Ramsey graphs, for pseudorandom generators in computer science, and for other instances of ``finding hay in a haystack''. 

Beyond its inherent appeal, this question is also closely related to
important notions in computational complexity theory and cryptography,
especially that of computational indistinguishability, introduced by
Goldwasser and Micali in \cite{gm}. Roughly speaking, one says that two probability distributions are \emph{computationally indistinguishable} if no efficient algorithm can distinguish samples from them (see \cite{avi}*{\S7.3.1} for a more thorough introduction). As such, one way of formalizing pseudorandomness, which is very influential in theoretical computer science, is to say that a construction is pseudorandom if it is computationally indistinguishable from the uniformly random distribution. 

Of course, proving computational indistinguishability amounts to proving limitations on the power of efficient algorithms, and hence is tantamount to proving that $P \neq NP$. Nonetheless, there is a great deal of interest in proving that certain restricted types of algorithms cannot distinguish an explicit pseudorandom object from a truly random one. For example, when it comes to distinguishing an explicit graph from a random graph, natural things one can try are to compute the spectrum and to count the number of copies of any constant-sized graph $H$, as both of these tasks can be easily done in polynomial time. Both of these are special cases of so-called \emph{low-degree tests}, which are those quantities that can be read off from a low-degree polynomial of the input (in this case, the entries of the adjacency matrix). In fact, it is not hard to see that in the setting of graphs, low-degree polynomials are equivalent to \emph{signed} subgraph counts (see, e.g., \cite{bb}*{Th.\,2.1}). Moreover, the \emph{low-degree heuristic} (see, e.g., \cite{kwb}*{\S\,4}) roughly states that such low-degree tests capture the entire notion of computational indistinguishability, in that two distributions are indistinguishable if and only if they agree on all low-degree tests.

Our construction gives a family of graphs that are easily
distinguishable from random graphs (by checking the presence of
$K_{2,3}$ as a subgraph), and yet are \emph{spectrally
  indistinguishable}. This implies that if one wants to prove
indistinguishability, it is not enough to restrict oneself to spectral
algorithms: some graphs that are easily distinguished from random graphs may
still have a completely random-like spectrum. One can compare the
situation to that of the \emph{planted clique problem}, which has become
a cornerstone problem in indistinguishability, where the best known
distinguishing algorithm \cite{aks} uses only spectral information.

Another way of looking at the same statement is to ask which structural features of a graph are determined by its spectrum. For example, the fact that the counts of short cycles determine the moments of the spectral distribution implies that if two graphs have asymptotically equal spectra, then they have roughly equal numbers of all short cycles. One could ask if something stronger is true: are the counts of all small subgraphs controlled by the spectrum? Our result shows that the answer is negative: our graphs are spectrally indistinguishable from a random graph of the same density, yet have very different numbers of copies of $K_{2,3}$. 

In the study of random matrices, semicircular limiting distributions
often come with other desirable properties, including \emph{eigenvector
  delocalization} and \emph{eigenvalue repulsion}. For a detailed
introduction to these topics, see the survey \cite{tv} of Tao and
Vu. Our graphs~$\Gamma_K(k)$ and~$\Gamma_B(k)$ exhibit essentially optimal
eigenvector delocalization, meaning that the $\ell^\infty$ norm of every
$(\ell^2$-normalized) eigenvector is $O(n^{-1/2})$, where $n$ is the
number of vertices of the graph. This follows immediately from the
algebraic structure of the graph; \cref{pr-1} below shows that the
eigenvectors of $\Gamma_K(k)$ and $\Gamma_B(k)$ are linear combinations of at
most~$2$ characters of the abelian group $k\times k$, which immediately
implies that every entry of the eigenvector is at most $O(n^{-1/2})$ in
absolute value. On the other hand, our graphs are extremely far from
exhibiting eigenvalue repulsion; it again follows immediately from
\cref{pr-1} below, as well as the symmetry properties
$K(a,b;k)=K(a\alpha,b\alpha^{-1};k)$ and~$B(a,b;k)=B(a\alpha,b\alpha^3;k)$ of Kloosterman and Birch sums (for
$\alpha$ non-zero), that every non-trivial eigenvalue of
$\Gamma_K(k)$ and~$\Gamma_B(k)$ appears with very high multiplicity, namely
multiplicity at least $\abs k-1 = (1+o(1))\sqrt n$ (in the case
of~$\Gamma_K(k)$). Again, these facts raise interesting questions about
how much the phenomena of eigenvector delocalization and eigenvalue
repulsion have to do with the limiting distribution itself.

\subsection{History}
Despite its natural appeal, there seems to have been very little prior
work on this question, and we are only aware of a few prior results.
\begin{enumerate}
\item First, McKay~\cite{mckay} determined the limiting spectral
  distribution for \emph{any} sequence of~$d$\nobreakdash-regular graphs, so long
  as $d$ is fixed and the girth\footnote{In fact, he proved that it
    suffices 
    that the number of short cycles is not too large.} of the graphs
  tends to infinity. Indeed, under such assumptions, one can
  explicitly compute the moments of the spectrum, as this boils down
  to counting rooted trees by the girth condition. The limiting
  distribution is the so-called \emph{Kesten--McKay distribution},
  given by the density function\footnote{This is not the most standard
    form of the Kesten--McKay distribution which appears in the
    literature; here we are renormalizing by $\sqrt{d-1}$ so that all
    of our limiting distributions are supported on $[-2,2]$. As
    pointed out by Serre~\cite{serre}*{p.\,80} in a similar context,
    in the case where $d=p+1$ for some prime number~$p$, this
    distribution is related to the Plancherel measure for the
    group~$\PGL_2(\Qq_p)$.}
  \[
    \frac{d(d-1)\sqrt{4-x^2}}{2\pi(d^2-(d-1)x^2)}
  \]
  for $x \in [-2,2]$. As $d \to \infty$, this density function
  approaches that of the semicircle distribution. As a consequence, if
  we have a sequence of regular graphs whose girth tends to infinity and
  whose degrees tend to infinity sufficiently slowly, then the limiting
  spectral distribution will be semicircular. Explicit estimates on the
  required relations between the order, degree and girth of the graphs
  are given by Sunada \cite{sunada} (also quoted in \cite{ms}*{Th.\,4})
  and Dumitriu--Pal\footnote{\ While
    \cite{dp}*{Th.\,1} is only stated for \emph{random} regular
    graphs, their proof uses nothing more than the control on the number of
    short cycles in these graphs.} \cite{dp}*{Th.\,1}.

  A natural family of explicit graphs satisfying (essentially) these
  conditions are certain Cayley graphs of the symmetric group $S_n$,
  namely those generated by the  transpositions
  $\{(12),(13),\dots,(1n)\}$. The fact that the spectrum of this
  family converges to the semicircle distribution was first obtained
  by Biane~\cite{biane}, by computing the moments and (implicitly)
  using the fact that these Cayley graphs have few short cycles. For
  an exposition of the proof, as well as more on the spectrum of these
  graphs and their history, see the note of Chapuy and Féray~\cite{cf}.

  For another example, the famous Ramanujan graphs of Lubotzky, Phillips
  and Sarnak~\cite{lps} give an explicit family of $d$-regular graphs
  (for $d=p+1$, where~$p$ is a prime number) whose girth tends to
  infinity with the order of the graph. As such, an appropriately chosen
  sequence of such Ramanujan graphs would yield an explicit sequence of
  graphs whose limiting spectral distribution is semicircular. However,
  note that in both of these examples, the degree of the graphs grows
  very slowly with their order, and this appears to be necessary to use such an approach (e.g., in~\cite{dp}*{Th.\,1}, the
  degree~$d_n$ must be $n^{o(1)}$, for a graph with~$n$ vertices). By
  contrast, our \cref{th-main} gives a degree that grows as fast as
  $\sqrt{n}$.
  
\item At the other extreme, this question was considered for graphs of
  \emph{linear} degree by Soloveychik, Xiang and Tarokh \cite{sxt}, in
  the equivalent guise of constructing explicit symmetric matrices
  with $\pm 1$ entries whose limiting spectrum is semicircular. In our
  language, they construct, for all $n=2^m-1$ , a family
  $\mathcal{G}_n$ of $n$ explicit graphs with~$n$ vertices, with the
  property that if one samples a graph~$\Gamma_n$ from $\mathcal G_n$
  uniformly at random, then almost surely the sequence $(\Gamma_n)$
  has the semicircular distribution as its limiting spectral
  distribution. Due to this random sampling, their construction is not
  fully explicit, although they conjecture that their result can be
  strengthened so that it holds for all choices of
  $\Gamma_n \in \mathcal G_n$ (and hence would give an explicit
  construction).
\end{enumerate}

In both of these examples, the convergence to the semicircular
distribution is proved using the method of moments: the moments of the
spectrum can be explicitly computed in terms of combinatorial counts of
closed walks in the graphs, and one can prove that these moments
converge to those of the semicircle law.

By contrast, our technique is completely different, and uses the fact
that the semicircular distribution appears in an entirely unrelated
setting. Namely, if one picks a matrix in~$\SU_2(\Cc)$ at random
according to the Haar probability measure, then its trace is distributed
according to the semicircle law (this follows from the Weyl integration
formula; see, e.g., \cite{bourbaki}*{p.\,339,\,Example}). This
distribution arises this way very naturally in number theory, where it
is called the \emph{Sato--Tate distribution} associated to $\SU_2(\Cc)$
(see Sutherland's survey~\cite{sutherland} for an introduction to this topic). To the best of our knowledge, there is no
direct connection between the spectrum of Wigner matrices and the traces
of $\SU_2(\Cc)$-random matrices, apart from the fact that they both
happen to have the same distribution.

We prove that the spectrum of our graphs converges to the semicircular
distribution by exploiting this coincidence. Specifically, we use
Deligne's \emph{equidistribution theorem}~\cite{deligne}*{\S\,3.5},
which allows us to relate the spectrum of the graphs to the traces of
certain matrices, which behave like Haar-distributed random matrices,
and results of Katz~\cite{gkm} which show that these matrices are in
$\SU_2(\Cc)$.

\subsection{The independence number}\label{subsec:independence}

Recall that given two graphs~$H_1$ and~$H_2$, the \emph{Ramsey
  number}\footnote{For an introduction to Ramsey theory, see the survey of Conlon, Fox and Sudakov~\cite{cfs} or the lecture notes of Wigderson~\cite{yuvalnotes}.}  $r(H_1,H_2)$
is the smallest integer~$N$ such that, for any partition of the edges of the
complete graph $K_N$ on $N$ vertices in two parts $E_1$ and~$E_2$, there
is either a copy of~$H_1$ in $E_1$ or one of~$H_2$ in $E_2$. In
particular, if $H_2=K_t$ for some integer~$t\geq 2$, then it follows
elementarily that an integer~$n$ is less than $r(H_1,K_t)$ if and only if there is
a graph with $n$ vertices without any copy of~$H_1$ and with independence
number less than~$t$. (We recall that the \emph{independence number} $\alpha(\Gamma)$ of a graph $\Gamma$ is the maximum size of a set of vertices containing no edges.)

Let~$C_4$ denote the $4$-cycle. One of the central open problems in
graph Ramsey theory is the estimation of the Ramsey number $r(C_4,K_t)$,
which is therefore equivalent to asking for the smallest independence
number among all $n$-vertex $C_4$-free graphs. As an $n$-vertex
$C_4$-free graph has $O(n^{3/2})$ edges by the Kővári--Sós--Turán
theorem~\cite{kst}, one immediately finds (see, e.g., \cite{aks80}) that
it must have independence number at least $c\sqrt n$, for some absolute
constant $c>0$. An infamous conjecture of Erdős~\cite{erdos1},
reiterated many times, posits that in fact there should exist some
absolute constant~$\delta>0$ such that
$\alpha(\Gamma) \geq \delta n^{\frac 12 + \delta}$ for every $n$-vertex
$C_4$-free graph $\Gamma$. To date, the best known lower bound is
$\alpha(\Gamma)\geq c\sqrt n \log n$ for some constant~$c>0$, which
follows immediately from the classical
Ajtai--Komlós--Szemerédi~\cite{aks80} bound on the independence number
of graphs with few triangles (note that in a $C_4$-free graph, every
edge lies on at most one triangle).

In the other direction, the best known upper bound for $\alpha(\Gamma)$
is $O((n\log n)^{\frac 23})$, which follows from work of
Bohman--Keevash~\cite{bk} on the random $C_4$-free process (see
also~\cite{mv} for a very different construction of Mubayi and
Verstraëte achieving the same bound). The main difficulty in closing the
gap between the lower and the upper bound seems to be the following: any
truly ``random-like'' $C_4$-free graph cannot have average degree much
greater than $n^{\frac 13}$, and hence cannot have independence number
much smaller than $n^{\frac 23}$. While there do exist a plethora of
constructions of denser $C_4$-free graphs, of average degree as large as~$\sqrt n$, all known constructions have a much larger independence
number, namely on the order of~$n^{\frac 34}$, than what we would expect
in a random graph of the same density.

Why $n^{\frac 34}$? It follows immediately from the expander mixing
lemma (see \eqref{eq:ind set} below) that if an $n$-vertex $d$-regular
graph has all of its non-trivial eigenvalues bounded by $\lambda$, then the bound
$\alpha(\Gamma) \leq n \lambda/d$ holds; a slightly more precise bound is
given by Hoffman's ratio bound (see the account by
Haemers~\cite{hoffman}). Known constructions of dense $C_4$-free graphs
are optimally spectrally pseudorandom, thus satisfying
$\lambda = O(\sqrt d) = O(n^{\frac 14})$. This immediately implies that
their independence number is at most $O(n^{\frac 34})$; however, while
there is no reason to expect this spectral bound to be tight in general,
it turns out to be so for all known examples of dense $C_4$-free graphs
(see the work of Mubayi and Williford~\cite{mw}).

In principle, more detailed knowledge of the whole spectrum, rather than
simply a bound on the magnitude of the non-trivial eigenvalues, should
allow one to obtain tighter control on the independence number. Indeed,
let $\Gamma$ be an $n$-vertex $d$-regular graph, let $A$ be its adjacency
matrix, and let $A = \sum_{i=1}^n \lambda_i v_i v_i^t$ be the spectral
decomposition of $A$, where $\lambda_1=d$ and
$v_1 = (\frac{1}{\sqrt n},\dots,\frac{1}{\sqrt n})$. If $I$ is any
independent set in $\Gamma$, then its indicator vector $\mathbf 1_I$ satisfies
\begin{equation}\label{eq:ind set}
  0=\langle \mathbf 1_I, A \mathbf 1_I\rangle =
  \sum_{i=1}^n \lambda_i \langle \mathbf 1_I,v_i\rangle^2 =
  \frac{d\abs I^2}{n}+\sum_{i=2}^n \lambda_i
  \langle \mathbf 1_I, v_i \rangle^2.
\end{equation}

By plugging in the bound $\abs{\lambda_i} \leq \lambda$ and the formula
$\sum_{i=1}^n \langle \mathbf 1_I,v_i\rangle^2=\abs I$, one immediately
obtains the bound $\abs I \leq n\lambda/d$ mentioned above. However,
this is in principle quite wasteful: we may expect to have a great deal
of cancellation in the sum
$\sum \lambda_i \langle \mathbf 1_I,v_i\rangle^2$, which could yield a
much stronger bound. Concretely, if the bound
$\alpha(\Gamma) \leq n \lambda/d$ is essentially tight, it means that a
largest independent set in $\Gamma$ must correlate very strongly with the
eigenvectors corresponding to the most negative eigenvalues of $\Gamma$.

In some instances, one can use such an idea to slightly improve on this
simple bound (see, e.g., Newman's thesis~\cite{newman}*{\S6.11}),
and there are a few other examples (e.g. the paper~\cite{etz} of
Elphick, Tang and Zhang) of using the entire spectrum to bound the
independence number. Moreover, it is very tempting to believe that such
ideas may be relevant for the case of~$C_4$\nobreakdash-free graphs: the classical
constructions of dense $C_4$-free graphs are strongly regular, and thus
their most negative eigenvalue has enormous multiplicity (roughly~$n/2$). It is thus relatively easy to construct an indicator vector of a
set that correlates very strongly with these eigenvectors.

This brings us back to our graphs. Since their limiting spectral
distribution is continuous, we are hopeful that their independence
number is substantially smaller than that of the classical
constructions: we believe that this continuous spectrum should yield
strong cancellation in \eqref{eq:ind set}. That said, proving this is
likely to be quite difficult: even proving that sets of the form
$I \times J$, where $I,J \subset \Ff_p$ are intervals, are not
independent in $\Gamma_K(\Ff_p)$
if~$\abs I \abs J \geq p^{\frac 32-\delta}$ is a difficult open
problem (see Theorem 13 and the subsequent discussion in the survey of
Shparlinski \cite{shparlinski}). While our graphs are not $C_4$-free,
they ``almost'' are.  Moreover, we can extract well-behaved $C_4$-free
subgraphs of them in such a way that the limiting spectral
distribution remains an explicit distribution with continuous density
(although it is not semicircular anymore; see Section
\ref{sec:subsets}).  Alternatively, one can simply
study $r(K_{2,3},K_t)$, or equivalently the minimum independence
number of $K_{2,3}$-free graphs, a problem with a great deal of
similarities with the $C_4$-free problem discussed above, and where
our graphs can be directly applied (see, e.g., the paper~\cite{clrz}
of Caro, Li, Rousseau and Zhang).

\section{Proof of Theorem \ref{th-main}}\label{sec:proof}

Although our proof will be quite short, the spectral information relies
on very deep results of algebraic geometry, namely 
\emph{Deligne's equidistribution theorem}, and on Katz's concrete
versions of this result for the specific case of Kloosterman sums and
Birch sums.

On the other hand, the fact that~$\Gamma_K(k)$ and~$\Gamma_B(k)$ do
not contain $K_{2,3}$ is straightforward.

More generally, let $G$ be an abelian group and $S$ a subset of~$G$. We
define the \emph{Cayley sum graph} $\Gamma(G,S)$ to have vertex set~$G$ and and edge
joining $x$ to~$y$ if and only if~$x+y\in S$.  We recall that a subset
$S\subset G$ is called a \emph{Sidon set} if the equation
\begin{equation}\label{eq-sidon}
  \alpha+\beta=\gamma+\delta,\quad\quad (\alpha,\beta,\gamma,\delta)\in S^4, 
\end{equation}
only has solutions with $\alpha\in \{\gamma,\delta\}$. Moreover, we say
that $S$ is a \emph{partial symmetric Sidon set}, with center some
element~$a_0\in S$, if the equation~(\ref{eq-sidon}) has only solutions
with $\alpha\in\{\gamma,\delta\}$ or~$\alpha+\beta=\gamma+\delta=a_0$
(in which case $\{\alpha,\gamma\}\subset S\cap (a_0-S)$). If the set~$S$
also satisfies $S=a_0-S$, then we say that $S$ is a \emph{symmetric
  Sidon set} with center~$a_0$ (as in~\cite{ffk2}).

We then have the following simple result, variants of which are
well-known in the literature.

\begin{proposition}\label{prop:K23}
  Let $G$ be an abelian group and $S$ a subset of~$G$.

  \begin{enumth}
  \item If $S$ is a Sidon set in~$G$, then $C_4\not\subset\Gamma(G,S)$.

  \item If $S$ is a partial symmetric Sidon set in~$G$, then
    $K_{2,3}\not\subset\Gamma(G,S)$.
  \end{enumth}
\end{proposition}

\begin{proof}
  Let $(w,x,y,z)$ be the vertices of a four-cycle in $\Gamma(G,S)$. We define
  $\alpha=w+x$, $\gamma=x+y$, $\beta=y+z$ and $\delta=z+w$, and note that these are elements
  of~$S$ such that
  \[
    \alpha+\beta=\gamma+\delta.
  \]

  (1) If~$S$ is a Sidon set, then we have either $\alpha=\gamma$, and hence
  $w=y$, or $\alpha=\delta$, and hence~$x=z$; both of these are
  impossible.

  (2) If~$S$ is a partial symmetric Sidon set with center~$a_0\in G$,
  then the only remaining option is that
  $\alpha+\beta=\gamma+\delta=a_0$ with $\alpha$,
  $\gamma\in S\cap (a_0-S)$. Given the values of $w\in G$
  and~$(\alpha,\gamma)\in S^2$, this gives \emph{at most one} $4$-cycle
  of the form
  \[
    (w,\alpha-w,\gamma-\alpha+w, a_0-\gamma-w)
  \]
  with opposite vertices $w$ and $y=\gamma-\alpha+w$. We conclude that there is no copy of $K_{2,3}$, either because every
  pair of vertices $(w,y)$ are joined by at most two distinct paths of
  length~$2$, or because $K_{2,3}$ contains two different $4$-cycles
  with opposite vertices.
 \end{proof}

Thus, the first statement of Theorem~\ref{th-main} follows from the
elementary observation that~$K(k)$ and~$B(k)$ (the latter for a field $k$ of 
characteristic different from~$3$) are symmetric Sidon sets with
center~$0$ in~$k\times k$, as stated in the following proposition.

\begin{proposition}\label{prop:symmetric sidon}
  For every finite field $k$, the set $K(k)$ is a symmetric Sidon set in
  $k \times k$ (with center $0$). If the characteristic of~$k$ is
  $\not=3$, then the same property holds for $B(k)$.
\end{proposition}

We remark that these facts appear implicitly in several places in the
literature, e.g.\ both in the paper \cite{ruzsa}*{p.\,349--350} of
Ruzsa, in a paper~\cite{weil2lectures}*{p.\,118} of Katz for $B(k)$, and
in Kloosterman's work~\cite{kloosterman}*{p.\,425} for~$K(k)$, but as
far as we are aware, not explicitly, although they are related to many of the known constructions of Sidon sets
(see, e.g., \cite{e-m}).

\begin{proof}[Proof of \cref{prop:symmetric sidon}]
  Let us fix some $(u,v) \in k \times k$. We wish to count the number of
  solutions to $(x,1/x)+(y,1/y) = (u,v)$ for $(x,1/x),(y,1/y) \in K(k)$;
  it suffices to show that if $(u,v)\neq (0,0)$, then there is at most
  one solution up to ordering. If $v=0$, then we have~$1/x=-1/y$, hence
  $x=-y$, and hence $u=0$ as well, so we may assume that $v \neq 0$. The
  equation $1/x+1/y=v$ is equivalent to $x+y=vxy$, which in turn implies
  that $xy=u/v$ using that $x+y=u$. That is, from the equation
  $(x,1/x)+(y,1/y)=(u,v)$, we have deduced that $x+y=u$ and
  $xy=u/v$. But we now know the sum and product of $x$ and $y$, which
  means that $x,y$ are unique up to ordering, as claimed. This proves
  the statement for $K(k)$.

  We now assume that~$k$ has characteristic~$\not=3$, and prove the
  claim for $B(k)$.  Again, given~$(u,v)\in k\times k$, we wish to
  determine the solutions to $(x,x^3)+(y,y^3)=(u,v)$. As before, if
  $u=0$ then $v=0$ as well, so we may assume $u \neq 0$. Plugging $y=u-x$ into the equation~$x^3+y^3=v$, we find that
  $x^3+(u-x)^3=v$, or equivalently that $3ux^2-3u^2x+u^3=v$. This is a
  quadratic equation in $x$ with non-zero leading coefficient (by the
  assumptions), and hence there are at most two values of $x$ satisfying
  this equation. Given $x$, the value of~$y=u-x$ is uniquely determined
  (and is in fact the other root of the quadratic equation), which
  completes the proof.
\end{proof}

To prove the remainder of the theorem, we require the next property,
which is also well-known. We denote here by~$\widehat{G}$ the
character group of a finite group~$G$.

\begin{proposition}\label{pr-1}
  Let $G$ be a finite abelian group and $S\subset G$ a non-empty subset.
  
  \begin{enumth}
  \item The spectrum of the Cayley sum graph~$\Gamma(G,S)$ consists of the
    numbers
    \[
      \sum_{y\in S}\chi(y)
    \]
    for all characters~$\chi\in \widehat{G}$ such that $\chi^2=1$ and
    of the numbers
    \[
      -\Bigl|\sum_{y\in S}\chi(y)\Bigr|,\quad\quad \Bigl|\sum_{y\in
        S}\chi(y)\Bigr|
    \]
    for all pairs $(\chi,\bar{\chi})$ of non-real characters.
  \end{enumth}

If some of the sums appearing in these statement coincide, for
different choices of~$\chi$, then the corresponding eigenvalue occurs
with multiplicity equal to the number of characters that give rise to
it.
\end{proposition}

\begin{proof}
  If~$A$ denotes the adjacency matrix, viewed as a linear map acting
  on the space~$\mathcal{C}$ of complex-valued functions on~$G$, we
  simply observe that by direct computation, we have
  \[
    A(\chi)=S(\chi)\bar{\chi}
  \]
  for any character~$\chi$ of~$G$, where
  \[
    S(\chi)=\sum_{y\in S}\chi(y).
  \]

  It follows that any real character~$\chi$ is an eigenvector of~$A$
  with eigenvalue $S(\chi)$, and that  any pair $(\chi,\bar{\chi})$
  of distinct non-real characters spans a $2$-dimensional invariant
  subspace on which~$A$ acts according to the matrix
  \[
    \begin{pmatrix}
      0& S(\bar{\chi})\\
      S(\chi) & 0
    \end{pmatrix}.
  \]

  The characteristic polynomial of this matrix is~$X^2-|S(\chi)|^2$, and hence its
  eigenvalues are the numbers $|S(\chi)|$ and $-|S(\chi)|$.
\end{proof}

\begin{proof}[Proof of Theorem~\ref{th-main}]
  We first assume that the characteristic of~$k$ is not~$2$.  We apply Proposition \ref{pr-1} to~$G=k\times k$ and~$S=K(k)$ or~$B(k)$. Fixing a
  non-trivial additive character~$\psi$ of~$k$, the characters
  of~$k\times k$ are of the form
  \[
    (x,y)\mapsto \psi(ax+by)
  \]
  for $(a,b)\in k\times k$, and none of them is a real character, with
  the exception of the trivial one. Thus the eigenvalues of the Cayley
  sum graphs $\Gamma_K(k)$ and $\Gamma_B(k)$ are, on the one hand the
  trivial eigenvalue, equal to the degree $|K(k)|$ or $|B(k)|$ of the
  graph, and on the other hand the quantities
  \[
    |K(a,b;k)| \quad \text{and}\quad -|K(a,b;k)|,\quad\quad\text{ or }\quad\quad
    |B(a,b;k)|\quad\text{and}\quad -|B(a,b;k)|,
  \]
  in terms of the exponential sums
  \[
    K(a,b;k)=\sum_{x\in k^{\times}}\psi(ax+bx^{-1}),\quad\quad
    B(a,b;k)=\sum_{x\in k}\psi(ax+bx^3).
  \]
    
  In the number theory literature, these are known as \emph{Kloosterman
    sums} or \emph{Birch sums}, respectively. They are real numbers, so
  that 
  \[
  K(-a,-b;k)=K(a,b;k)\quad \text{and} \quad B(-a,-b;k)=B(a,b;k).
  \]
 
  A well-known application of the Riemann hypothesis for curves over
  finite fields, due to Weil, implies the bounds
  \[
    |K(a,b;k)|\leq 2\sqrt{|k|},\quad\quad |B(a,b;k)|\leq 2\sqrt{|k|},
  \]
  if $(a,b)\not=(0,0)$ (see, e.g.,~\cite{ik}*{\S\,11.7} for an
  elementary proof using Stepanov's method, in the case of Kloosterman
  sums), thus proving the second statement of Theorem~\ref{th-main}.
  
  The final and deepest statement (for which no fully elementary proof
  is known) is provided by work of Katz.  In the case of Kloosterman
  sums, Katz proved in~\cite{gkm}*{Th.\,11.1,\,Th.11.4} that the
  normalized Kloosterman sums $K(a,1;k)/\sqrt{|k|}$ for
  $a\in k^{\times}$ become equidistributed according to the semicircle
  distribution as $|k|\to+\infty$.  The relation
  \[
    K(a,b;k)=K(ab,1;k)
  \]
  for $(a,b)\in k^{\times}\times k^{\times}$ implies that the
  distribution of $K(a,b;k)/\sqrt{|k|}$ (as a measure on $\Rr$) is the
  same as that of $K(a,1;k)/\sqrt{|k|}$. Moreover, since
  $K(a,0;k)=K(0,b;k)=-1$ for~$ab\not=0$, the contribution of the
  eigenvalues with $a=0$ or $b=0$ (but not both) to the limiting
  distribution vanishes.

  Since the measure~$\mu_{\mathrm{sc}}$ is symmetric (i.e., invariant under
  the transformation $x\mapsto -x$ on~$[-2,2]$), it follows that the
  same equidistribution properties hold for the absolute values of the
  sums and their opposite, which give the non-trivial spectrum of the
  Cayley sum graphs for odd characteristic. Precisely, the empirical
  spectral measure for the non-trivial eigenvalues of~$\Gamma_K(k)$ is
  \[
    \frac{1}{|k|^2-1}\sum_{\text{pairs }(a,b)}
    \Bigl(\delta_{K(a,b;k)/\sqrt{|k|}}+
    \delta_{-K(a,b;k)/\sqrt{|k|}}\Bigr),
  \]
  where the sum ranges over pairs $((a,b),(-a-b))$ of non-trivial
  characters, whereas the empirical measure for Kloosterman sums is
  \[
    \frac{1}{|k|^2-1}\sum_{(a,b)} \delta_{K(a,b;k)/\sqrt{|k|}}=
    \frac{2}{|k|^2-1}\sum_{\text{pairs }(a,b)}
    \delta_{K(a,b;k)/\sqrt{|k|}},
  \]

  Since the latter converges to~$\mu_{\mathrm{sc}}$ by Katz's work,
  and since~$\mu_{\mathrm{sc}}$ is symmetric, the same applies to the
  former.

  The analogue of the bound~(\ref{eq-w1}) is proved
  in~\cite{k-u}*{Th.\,4.4} (as spelled-out
  in~\cite{k-u}*{Example\,4.5\,(1)}) for the distribution
  of~$K(a,1;k)/\sqrt{|k|}$ with~$a\in k^{\times}$. By the preceding
  remarks, this implies~(\ref{eq-w1}).

  In the case of Birch sums, Katz proved the equidistribution
  in~\cite{katz-birch}*{Th.\,19,\,Cor.\,20} for
  $a\mapsto B(a,b;k)/\sqrt{|k|}$ with $a$ varying in~$k$ for any fixed
  $b\in k^{\times}$, when~$k$ is restricted to fields of characteristic
  $\geq 7$ (the statement of loc. cit. suggests that it requires $p>7$;
  however, the result is valid for $p=7$ by inspection, the point being
  that the value of~$n$ in~\cite[Th.\,9]{katz-birch}, which is used in
  the proof of~\cite[Th.\,19]{katz-birch}, is the rank~$2$ and not the
  degree~$3$ of the polynomial).

  Since $B(a,0;k)=0$ for $a\in k^{\times}$, the values which are
  excluded do not affect the equidistribution, and the remainder of
  the argument is essentially identical.

  In characteristic~$2$, a case that only occurs for Kloosterman
  sums, the argument is similar but simpler: all characters are real
  characters, so the eigenvalues of the adjacency matrix are exactly the
  Kloosterman sums $K(a,b; k)$, and we can directly apply the results of
  Katz.
\end{proof}

Although this gives a complete proof of Theorem~\ref{th-main}, we
explain in the next section an alternate proof of the equidistribution
theorem in the case of $\Gamma_K(k)$ which is based on more recent ideas
and is more transparent. In fact, this will reveal that the
$K_{2,3}$-freeness of the graphs is closely related to their spectral
properties.

\section{Larsen's alternative, Sidon sets and equidistribution}\label{sec:larsen}

For simplicity, we will restrict our argument to the setting where $k$
varies among the extensions of a fixed base finite field. The starting
point is Deligne's general \emph{a priori} equidistribution theorem,
which relies on Deligne's most general form of the Riemann hypothesis
over finite fields. Specialized to the case of Kloosterman sums, this is
the following statement (most easily deduced from the version
in~\cite{ffk}*{Th.\,4.8}; see also~\cite{deligne}*{\S3.5} for the
original statement, and~\cite{gkm}*{Th.\,3.6} for another version).

\begin{theorem}
  Let~$k$ be a finite field.  There exist an integer~$r\geq 0$ and two
  compact Lie groups $K^g$ and~$K^a$, contained in $\Un_r(\Cc)$ such
  that the following properties hold:
  \begin{enumth}
  \item The subgroup $K^g$ is a normal subgroup of~$K^a$ with
    $K^a/K^g$ abelian.
    
  \item For any continuous and bounded function $f\colon \Cc\to \Cc$, we
    have
    \[
      \lim_{N\to +\infty}\frac{1}{N}\sum_{1\leq n\leq N}
      \frac{1}{|k_n^{\times}|^2}\sum_{a,b\in k_n^{\times}} f\Bigl(
      \frac{1}{\sqrt{|k_n|}}K(a,b;k_n) \Bigr) =\int_{K^a}f(\Tr(x))dx,
    \]
    where~$k_n$ is the extension of degree~$n$ of~$k$ in an algebraic
    closure of~$k$, and the character~$\psi_n=\psi\circ \Tr_{k_n/k}$
    is used to define~$K(a,b;k_n)$.
    
  \item If $K^a=K^g$, then for any continuous and bounded function
    $f\colon \Cc\to \Cc$, we have
    \[
      \lim_{n\to +\infty} \frac{1}{|k_n^{\times}|^2}\sum_{a,b\in
        k_n^{\times}} f\Bigl( \frac{1}{\sqrt{|k_n|}}K(a,b;k_n) \Bigr)
      =\int_{K^a}f(\Tr(x))dx.
    \]
  \end{enumth}

  In these statements, $dx$ denotes the unique probability Haar measure
  on the group~$K^a$.
\end{theorem}

Assuming this deep result, we can deduce the equidistribution theorem of
Katz for Kloosterman sums, provided we can show that $r=2$ and
$K^a=K^g=\SU_2(\Cc)$.  The first assertion that $r=2$ is a standard
consequence of the Grothendieck--Ogg--Shafarevich formula (see, e.g.,~\cite{gkm}*{\S\,2.3.1}). It corresponds also to the
following fundamental property: according to the formalism underlying
Deligne's equidistribution theorem, and the constructions of Katz, it is
known that for all $(a,b)\in (k^{\times})^2$, there exists an
element~$\theta(a,b;k)$ of~$\Un_2(\Cc)$ such that
\[
  \Tr(\theta(a,b;k))=\frac{1}{\sqrt{k}}\sum_{x\in
    k^{\times}}\psi(ax+bx^{-1}). 
\]

This fact explains in particular the bound~$|K(a,b;k)|\leq
2\sqrt{|k|}$, as all elements of $\Un_2(\Cc)$ have trace in $[-2,2]$. Moreover, one proves that the characteristic polynomial of
$\theta(a,b;k)$ is
\[
  X^2-\frac{K(a,b;k)}{\sqrt{|k|}}X+1,
\]
which implies that $\theta(a,b;k)\in\SU_2(\Cc)$, and then one deduces
that $K^a\subset \SU_2(\Cc)$.

\begin{lemma}\label{lm-larsen}
  With notation as above, either $K^a=K^g=\SU_2(\Cc)$ or $K^a$ is a
  finite subgroup of~$\SU_2(\Cc)$ isomorphic to one of $\SL_2(\Ff_3)$,
  $\GL_2(\Ff_3)$ or $\SL_2(\Ff_5)$ embedded in $\SU_2(\Cc)$ by an
  irreducible faithful representation.
\end{lemma}

\begin{proof}
  The fact that $K^a$ is contained in $\SU_2(\Cc)$ has been explained above. Next, by equidistribution again, we can compute
  \[
    M_4=\int_{K^a}\abs{\Tr(x)}^4dx=2
  \]
  using the fact that the sets $K(k_n)$ are symmetric Sidon sets.
  More precisely, it follows from equidistribution that
  \[
    M_4=\lim_{N\to+\infty} \frac{1}{N}\sum_{n\leq N}
    \frac{1}{|k_n^{\times}|^2}\sum_{a,b\in k_n^{\times}}
    \Bigl|\frac{1}{\sqrt{|k_n|}} \sum_{x\in
      k_n^{\times}}\psi_n(ax+bx^{-1}) \Bigr|^4.
  \]

  Expanding the fourth power, the average over $(a,b)$ on the right-hand side, for a given $n\geq 1$, is equal to
  \[
    \frac{1}{|k_n|^2}\sum_{x_1,\ldots,x_4\in k_n^{\times}}
    \frac{1}{|k_n|^{\times}}\sum_{a\in k_n^{\times}}
    \psi_n(a(x_1+x_2-x_4-x_4)) \frac{1}{|k_n|^{\times}} \sum_{b\in
      k_n^{\times}}\psi_n(b(x_1^{-1}+x_2^{-1}-x_3^{-1}-x_4^{-1})).
  \]

  Writing both sums over~$a$ and~$b$ as sums over~$k_n$ minus the
  contributions of~$a=0$ and~$b=0$, we obtain four terms. The first
  one, by orthogonality of characters, is
  \[
    \frac{1}{|k_n|^2}\sum_{\substack{x_1,\ldots,x_4\in
        k_n^{\times}\\x_1+x_2=x_3+x_4\\x_1^{-1}+x_2^{-1}=x_3^{-1}+x_4^{-1}}}
    1,
  \]
  which converges to~$3$ as~$n\to+\infty$ because $K(k_n)$ is a
  symmetric Sidon set.

  Using orthogonality again, the additional contributions are:
  \begin{gather*}
    -\frac{1}{|k_n|^3}
    \sum_{\substack{x_1,\ldots,x_4\in
        k_n^{\times}\\x_1^{-1}+x_2^{-1}=x_3^{-1}+x_4^{-1}}}
    1\quad\quad \text{(from $a=0$)},\\
    -\frac{1}{|k_n|^3}
    \sum_{\substack{x_1,\ldots,x_4\in
        k_n^{\times}\\x_1+x_2=x_3+x_4}}
    1\quad\quad \text{(from $b=0$)},\\
    \frac{1}{|k_n|^4}
    \sum_{x_1,\ldots,x_4\in
      k_n^{\times}}1\quad\quad \text{(from $a=b=0$)}.
  \end{gather*}

  The first two converge to~$-1$ each, and the last converges
  to~$1$. Altogether, we obtain
  \[
    M_4=\lim_{N\to+\infty}\frac{1}{N}\sum_{n\leq N}
    \frac{1}{|k_n^{\times}|^2}\sum_{a,b\in
      k_n^{\times}} \Bigl|\frac{1}{\sqrt{|k_n|}}
    \sum_{x\in k_n^{\times}}\psi_n(ax+bx^{-1})
    \Bigr|^4=3-1-1+1=2.
  \]
  
  By Larsen's alternative (see~\cite{katz-larsen}*{Th.\,1.1.6}
  or~\cite{ffk}*{Th.\,8.5}), the fact that $M_4=2$ implies that either
  $K^a=\SU_2(\Cc)$, or that~$K^a$ is a finite subgroup
  of~$\SU_2(\Cc)$. In the first case, we deduce that in fact
  $K^a=K^g=\SU_2(\Cc)$ because $K^g$ is a normal subgroup of~$K^a$
  with abelian quotient and $\SU_2(\Cc)$ has no such subgroup except
  itself.

  If $K^a$ is finite, then Katz has
  proved~\cite{katz-larsen}*{Th.\,1.3.2} that the condition $M_4=2$
  implies that the representation $K^a\subset \SU_2(\Cc)$ is not induced
  from a representation of a proper subgroup. This means that~$K^a$ is
  an irreducible primitive subgroup of~$\SU_2(\Cc)$. The classification
  of these subgroups is classical (it is a variant of the well-known
  classification of finite subgroups of $\SO_3$), and there are three
  possible groups, often called the binary tetrahedral, binary
  octahedral and binary icosahedral groups.

  To check that these three groups are indeed those we indicated, it is
  enough to check that the orders correspond and that $\SL_2(\Ff_3)$,
  $\GL_2(\Ff_3)$ and $\SL_2(\Ff_5)$ have faithful two-dimensional
  representations, since we know that $\SU_2(\Cc)$ contains no other
  finite primitive subgroup; we refer the reader, e.g,
  to~\cite{leuschke-wiegand}*{Th.\,6.11} for this classification, and
  the other facts are elementary.
\end{proof}

The following finishes the proof of the equidistribution of Kloosterman
sums for the sequence of fields $(k_n)_{n\geq 1}$.

\begin{lemma}
  The group~$K^a$ associated to Kloosterman sums is infinite.
\end{lemma}

\begin{proof}
  The key point is that the formalism implies that the matrix
  $\theta(a,b;k)\in \SU_2(\Cc)$ discussed above is in fact in $K^a$ for
  all $(a,b)\in k^{\times}\times k^{\times}$ (and is unique up to
  conjugacy in~$K^a$).  Suppose then that~$K^a$ is finite. The trace of
  $\theta(a,b;k)$, as the trace of a matrix of finite order, is a sum of
  finitely many roots of unity, and hence is an algebraic
  integer. However, taking~$b=1$ and summing over $a\in k^{\times}$, we
  find that
  \[
    \sum_{a\in k^{\times}}\Tr(\theta(a,1;k))=
    \sum_{a\in k}\frac{1}{\sqrt{k}}\sum_{x\in
      k^{\times}}\psi(ax+x^{-1})
    -\frac{1}{\sqrt{|k|}}\sum_{x\in k^{\times}}\psi(x^{-1})=
    \frac{1}{\sqrt{|k|}}
  \]
  by orthogonality of characters.  Since this is not an algebraic
  integer, we have a contradiction.
\end{proof}

\begin{remark}\label{rm-finite}
  It seems difficult to find a similarly elementary argument for the
  case of Birch sums. The reason is that the restriction $p\geq 7$ on
  the characteristic is necessary, and hence any proof must
  necessarily involve this assumption in some way. More precisely,
  Katz proved (see~\cite[3.8.4]{mmp}
  or~\cite[Th.\,4.1,\,4.2]{katz-finite}) that for the Birch sums
  associated to fields of characteristic~$2$ or~$5$, the group~$K^a$
  is finite. This is depicted in \cref{fig:char 25}.
\end{remark}
\begin{figure}[ht]
  \includegraphics[width=.8\textwidth]{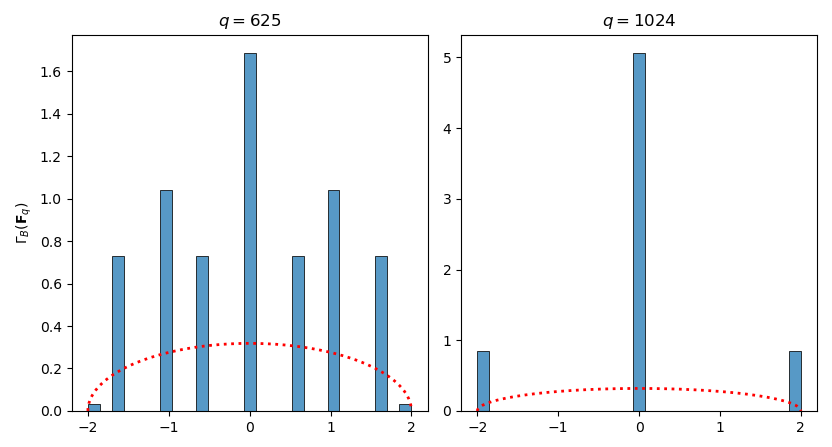}
  \caption{The spectrum of the graphs $\Gamma_B(\Ff_q)$ for $q = 625$ and $q=1024$. As predicted by Katz's theorem, the limiting spectrum is a finitely-supported atomic measure because the characteristic is less than $7$.}\label{fig:char 25}
\end{figure}

\section{Variants}\label{sec:subsets}

We describe here two types of variants, both obtained by considering
graphs with edges given by subsets of the hyperbola $K(k)$. In the first,
we present an \emph{uncountable} family of possible limiting spectral
distributions for $K_{2,3}$-free graphs of essentially optimal
density, as well as for $C_4$-free graphs. In the second case, another
example of $C_4$-free graphs is given.

We will use a simple observation in both cases.

\begin{lemma}\label{lem:desymmetrize}
  Let~$G$ be a finite abelian group and~$S \subset G$ a
  partial symmetric Sidon set with center~$0$. If~$T\subset G$ is a
  subset such that $T\cap (-T)=\emptyset$, then $S\cap T$ is a Sidon
  set.
\end{lemma}

\begin{proof}
  This is immediate from the definitions.
\end{proof}

Given a complex-valued random variable~$X$, we will denote
by~$\eps|X|$ the product of~$|X|$ with a random
variable~$\eps$ independent of~$X$ and taking values~$-1$ and~$1$ each
with probability~$1/2$. We will also use the additive
character of~$\Ff_p$ defined by
\[
\psi_p(x)=e(x/p),\quad \text{where}\quad e(z)=e^{2i\pi z}\quad \text{for}\quad z\in\Cc.
\]

Our first statement is the following.  

\begin{theorem}\label{th-klt}
  Let~$t$ be a real number with $0<t\leq 1$.  Let
  \[
    K_t(p)=\{(x,y)\in \Ff_p\times \Ff_p\,\mid\, xy=1\text{ and } 1\leq
    x\leq t(p-1)\}, 
  \]
  where we identify $\Ff_p$ with the integers $\{0,\ldots, p-1\}$.

  \begin{enumth}
  \item The graph $\Gamma(\Ff_p\times\Ff_p,K_t(p))$ is $K_{2,3}$-free,
    and $C_4$-free if~$0<t\leq 1/2$.

  \item All the non-trivial eigenvalues~$\lambda$ of the adjacency
    matrix of~$\Gamma(\Ff_p\times\Ff_p,K_t(p))$ satisfy
    \[
      |\lambda|=O(p^{1/2}\log p),
    \]
    where the implied constant is absolute, and in particular,
     independent of~$t$.
    
  \item As $p\to +\infty$ among primes, the numbers
    \[
      \Bigl\{\frac{\lambda}{\sqrt{|K_t(p)|}}\,\mid\, \lambda\text{ an
        eigenvalue of }\Gamma(\Ff_p\times\Ff_p,K_t(p))\Bigr\}
    \]
    converge in distribution to the distribution of
    \[
      \eps\Bigl|t^{1/2}\mathrm{SC}_0+\frac{1}{\sqrt{t}}
      \sum_{\substack{h\in\Zz\\h\not=0}}
      \frac{e(ht)-1}{2i\pi h}\mathrm{SC}_h\Bigr|, 
    \]
    where $(\mathrm{SC}_h)_{h\in\Zz}$ are independent
    semicircle-distributed random variables and the random series
    converges almost surely when taken as a limit of symmetric partial
    sums. 
  \end{enumth}
\end{theorem}

\begin{proof}
  The first assertion follows from the fact that~$K_t(p)$ is a partial
  symmetric Sidon set, and from \cref{lem:desymmetrize}, applied with
  $G=\Ff_p\times\Ff_p$, $S=K_{t}(p)$ with $0<t\leq 1/2$ and
  \[
    T=\Bigl\{(x,y)\in G\,\mid\, 1\leq x\leq \frac{p-1}{2}\Bigr\}.
  \]

  The second statement follows from the completion method
  for exponential sums (see, e.g.,~\cite{ik}*{\S12.2}), applied to
  \[
    \sum_{1\leq x\leq (p-1)t}\psi_p(ax+bx^{-1}).
  \]

  According to a result of Kowalski and
  Sawin~\cite{paths}*{Th.\,1.1\,(1)}, the normalized sums
  \[
    \frac{1}{\sqrt{p}}\sum_{1\leq x\leq (p-1)t}\psi_p(ax+bx^{-1})
  \]
  converge in distribution as $p\to+\infty$ to the sum
  \[
    t\mathrm{SC}_0+\sum_{\substack{h\in\Zz\\h\not=0}}
    \frac{e(ht)-1}{2i\pi h}\mathrm{SC}_h,
  \]
  and we deduce the third statement from this and
  Proposition~\ref{pr-1}, since~$|K_t(p)|\sim (p-1)t$.
\end{proof}

The second statement uses a different set~$T$ to obtain $C_4$-free
graphs.

\begin{theorem}\label{th-sidon}
  Let~$p$ be a prime number with $p\equiv 3\mods{4}$.
  Set
  \[
    K_+(p)=\{(x,y)\in \Ff_p\times \Ff_p\,\mid\, xy=1\text{ and }
    \Bigl(\frac{x}{p}\Bigr)=1\},
  \]
  where $(\tfrac{x}{p})$ is the Legendre symbol.

  \begin{enumth}
  \item The graph $\Gamma(\Ff_p\times\Ff_p,K_+(p))$ is a regular graph
    of degree $(p-1)/2$, and it is $C_4$-free.

  \item All the non-trivial eigenvalues~$\lambda$ of the adjacency
    matrix of~$\Gamma(\Ff_p\times\Ff_p,K_+(p))$ satisfy
    \[
      |\lambda|\leq 2p^{1/2}.
    \]
    
  \item As $p\to +\infty$ among primes that are~$\equiv 3\mods{4}$,
    the numbers
    \[
      \Bigl\{\frac{\lambda}{\sqrt{|K_+(p)|}}\,\mid\, \lambda\text{ an
        eigenvalue of }\Gamma(\Ff_p\times\Ff_p,K_+(p))\Bigr\}
    \]
    converge in distribution to the distribution of $\eps |SC+SA|$, where $(SA, SC)$ are
    independent random variables such that $SC$ is semicircle-distributed and $SA$ has distribution
    \[
      \mu_{\mathrm{sa}}=\frac{1}{2}\delta_0+\frac{1}{2}(x\longmapsto
      2i\cos(4\pi x))_*dx,
    \]
    with $dx$ the Lebesgue measure on~$[0,1]$.
  \end{enumth}
\end{theorem}

\begin{proof}
  To prove (1), we apply Lemma \ref{lem:desymmetrize} to
  $G=\Ff_p\times\Ff_p$, $S=K_+(p)$ and
  \[
    T=\{(x,y)\in\Ff_p\times\Ff_p\,\mid\, (\tfrac{x}{p})=1\}
  \]
  (using the fact that~$-1$ is not a square modulo~$p$).

  The characteristic function of~$K_+(p)$ inside $\Ff_p^{\times}$ is the function
  \[
    x\longmapsto \frac{1}{2}\Bigl(1+\Bigl(\frac{x}{p}\Bigr)\Bigr),
  \]
  and therefore the eigenvalues of the adjacency matrix are determined
  by the sums
  \[
    \frac{1}{2}\sum_{x\in \Ff_p^{\times}}\psi_p(ax+bx^{-1})+
    \frac{1}{2}\sum_{x\in
      \Ff_p^{\times}}\Bigl(\frac{x}{p}\Bigr)\psi_p(ax+bx^{-1})=
    \frac{1}{2}K(a,b;p)+\frac{1}{2}T(a,b;p),
  \]
  where the sums $T(a,b; p)$ are \emph{Salié sums}, and we write
  $K(a,b;p)$ instead of $K(a,b;\Ff_p)$ for convenience.

  As already mentioned, the Kloosterman sums are $\leq 2\sqrt{p}$ in
  absolute value, and the work of Katz (which was already used to prove
  Theorem~\ref{th-main}) shows that the (normalized) Kloosterman sums
  have semicircle limiting distribution.

  The case of Salié sums is also known but proceeds somewhat
  differently, because the ``monodromy group'' for Salié sums modulo~$p$
  (which governs the equidistribution properties for extensions
  of~$\Ff_p$ by Deligne's equidistribution theorem) is a finite group
  which depends on~$p$, in contrast with Kloosterman sums where it is $\SU_2(\Cc)$ for all~$p$.

  We use instead that Salié sums can be evaluated
  elementarily. Precisely, the formulas
  \begin{equation}\label{eq-salie}
  \begin{aligned}
    T(a,0;p)&=T(0,b;p)=-1\text{ if $a$, $b$} \not=0,\\
    T(a,b;p)&=
    \begin{cases}
      0&\text{ if } (\tfrac{ab}{p})=-1,\\
      2i\sqrt{p}(\frac{b}{p})\cos(4\pi y/p)&\text{ if } ab=y^2\not=0
    \end{cases}
  \end{aligned}
  \end{equation}
  hold for $p\equiv 3\mods{4}$ (see, e.g.,~\cite{iwaniec}*{Lemma\,4.9};
  the condition $p\equiv 3\mods{4}$ is used to evaluate the quadratic
  Gauss sum modulo~$p$).

  It follows\footnote{\ This can also be deduced from the general Weil
    bound.} first that $|T(a,b;p)|\leq 2\sqrt{p}$ for
  $(a,b)\not=(0,0)$, so that combined with the Weil bound, we conclude
  that the eigenvalues of the adjacency matrix have modulus
  $\leq 2\sqrt{p}$.

  We next claim that the normalized Salié sums $T(a,b;p)/\sqrt{p}$
  become $\mu_{\mathrm{sa}}$-equidistributed as~$p\to+\infty$ with
  $p\equiv 3\mods{4}$.  To see this, we write the ``empirical
  measure'' for the Salié sums $T(a,b;p)$ with $(a,b)\not=(0,0)$
  in the form
  \[
    \frac{1}{p^2-1}\sum_{(a,b)\not=(0,0)}\delta_{T(a,b;p)/\sqrt{p}}=
    \frac{1}{p^2-1}\sum_{ab=0}\delta_{T(a,b;p)/\sqrt{p}}+\mu_{1,p}.
  \]
  The first term converges weakly to the zero measure as
  $p\to+\infty$. On the other hand, according to~(\ref{eq-salie}), we
  get
  \[
    \mu_{1,p}=\frac{1}{p^2-1}\Bigl(
    \sum_{\substack{a,b\\(ab/p)=-1}}\delta_0+
    \sum_{\substack{a,b\\(ab/p)=1}}\delta_{2i(b/p)\cos(4\pi
      y/p)}\Bigr),
  \]
  where~$y$ denotes a square root of~$ab$ in~$\Ff_p^{\times}$. The
  first term in this sum converges to $\demi \delta_0$, since there
  are $(p-1)^2/2$ values of $(a,b)$ with $(\tfrac{ab}{p})=-1$. We
  express the second as
  \[
    \frac{1}{2(p^2-1)}\sum_{y\in\Ff_p^{\times}}
    \sum_{b\in\Ff_p^{\times}} \delta_{2i(b/p)\cos(4\pi y/p)}.
  \]

  Let~$f(x)=2i\cos(4\pi x)$ for $x\in [0,1]$.  Since~$y$ and~$b$ are
  independent variables, and since the~$y/p$ (resp.\,the Legendre symbols
  $(b/p)$) become uniformly distributed on $[0,1]$ (resp. become
  uniformly distributed on $\{-1,1\}$), the previous expression
  converges to the measure
  \[
    \frac{1}{2}\Bigl(\frac{1}{2}f_*dx+\frac{1}{2}(-f)_*dx\Bigr)
  \]
  as $p\to +\infty$, where $dx$ denotes the Lebesgue measure
  on~$[0,1]$. But the measures $f_*dx$ and~$(-f)_*dx$ are equal, and
  therefore the limit is $\demi f_*dx$. This confirms that the measure $\mu_{\mathrm{sa}}$ is the limiting
  measure for Salié sums.

  It remains to prove that the sets of pairs
  \[
    \Bigl(\frac{K(a,b;p)}{\sqrt{p}},\frac{T(a,b;p)}{\sqrt{p}}\Bigr)
  \]
  of normalized Kloosterman sums and Salié sums converge to a pair of
  \emph{independent} random variables. This is a problem of a relatively
  familiar kind, but this case does not seem to have been recorded in
  the literature.

  We restrict our attention to
  $(a,b)\in\Ff_p^{\times}\times \Ff_p^{\times}$, as we may for the
  same reasons as before.  By the method of moments (applicable
  because all measures here are compactly supported), it then suffices
  to prove that for all integers~$\alpha\geq 0$ and~$\beta\geq 0$, we
  have the limit
  \begin{equation}\label{eq-limit}
    \lim_{\substack{p\to+\infty\\p\equiv 3\mods{4}}} \frac{1}{(p-1)^2}
    \sum_{a,b\in\Ff_p^{\times}}
    U_{\alpha}\Bigl(\frac{K(a,b;p)}{\sqrt{p}}\Bigr)
    \Bigl(\frac{T(a,b;p)}{\sqrt{p}}\Bigr)^{\beta}=
    \expect(U_{\alpha}(SC))\expect(SA^{\beta}),
  \end{equation}
  where~$U_{\alpha}$ is the Chebychev polynomial corresponding to the
  character of the symmetric $\alpha$-th power representation
  of~$\SU_2(\Cc)$, so that
  \[
    \expect(U_{\alpha}(SC))=\begin{cases}
      1&\text{ if } \alpha=0,\\
      0&\text{ otherwise.}
    \end{cases}
  \]
         
  The distribution results of Kloosterman and Salié sums individually
  show that the formula holds if either~$\alpha=0$ or~$\beta=0$. Thus,
  we may assume that both $\alpha$ and~$\beta$ are non-zero, and in
  particular, the limit of the quantity in~(\ref{eq-limit}) should
  be~$0$.

  By~(\ref{eq-salie}), we obtain the expression
  \[
    \sum_{a,b\in\Ff_p^{\times}}
    U_{\alpha}\Bigl(\frac{K(a,b;p)}{\sqrt{p}}\Bigr)
    \Bigl(\frac{T(a,b;p)}{\sqrt{p}}\Bigr)^{\beta} =(2i)^{\beta}
    \sum_{(ab/p)=1} U_{\alpha}\Bigl(\frac{K(a,b;p)}{\sqrt{p}}\Bigr)
    \Bigl(\frac{b}{p}\Bigr)^{\beta}\cos\Bigl(\frac{4\pi
      y}{p}\Bigr)^{\beta},
  \]
  where again~$y$ denotes one square root of $ab$ modulo~$p$. Using
  the formula $K(a,b;p)=K(ab,1;p)$ and putting~$x=ab$, this (up to the
  factor $(2i)^{\beta}$) becomes
  \[
    \sum_{(x/p)=1} U_{\alpha}\Bigl(\frac{K(x,1;p)}{\sqrt{p}}\Bigr)
    \cos\Bigl(\frac{4\pi y}{p}\Bigr)^{\beta} \sum_{ab=x}
    \Bigl(\frac{b}{p}\Bigr)^{\beta},
  \]
  where~$y$ is now a square root of~$x$ modulo~$p$. The inner sum
  vanishes if~$\beta$ is odd, so that the limiting formula certainly
  holds in this case. We thus assume that~$\beta$ is even, in which case
  the inner sum is always equal to~$p-1$, so that the average
  in~(\ref{eq-limit}) is given by 
  \begin{align*}
    \frac{1}{(p-1)^2} \sum_{a,b\in\Ff_p^{\times}}
    U_{\alpha}\Bigl(\frac{K(a,b;p)}{\sqrt{p}}\Bigr)
    \Bigl(\frac{T(a,b;p)}{\sqrt{p}}\Bigr)^{\beta}
    &
      =\frac{1}{p-1}
      \sum_{(x/p)=1} U_{\alpha}\Bigl(\frac{K(x,1;p)}{\sqrt{p}}\Bigr)
      \cos\Bigl(\frac{4\pi y}{p}\Bigr)^{\beta}
    \\
    &=
    \frac{1}{2(p-1)}
    \sum_{y\in\Ff_p^{\times}} U_{\alpha}\Bigl(\frac{K(y^2,1;p)}{\sqrt{p}}\Bigr)
    \cos\Bigl(\frac{4\pi y}{p}\Bigr)^{\beta}
  \end{align*}
  (where~$y$ is a square-root of~$x$ in the middle step).

  The last transformation is to expand the cosine in complex
  exponentials, namely
  \[
    \sum_{y\in\Ff_p^{\times}}
    U_{\alpha}\Bigl(\frac{K(y^2,1;p)}{\sqrt{p}}\Bigr)
    \cos\Bigl(\frac{4\pi y}{p}\Bigr)^{\beta}= \frac{1}{2^{\beta}}
    \sum_{0\leq j\leq \beta}\binom{\beta}{j} \sum_{y\in\Ff_p^{\times}}
    U_{\alpha}\Bigl(\frac{K(y^2,1;p)}{\sqrt{p}}\Bigr)
    e\Bigl(\frac{2(2j-\beta)y}{p}\Bigr).
  \]

  We are now in a position to simply apply the (by now) fairly
  standard Lemma~\ref{lm-rh} below to conclude the proof.
\end{proof}

\begin{lemma}\label{lm-rh}
  For $\alpha\not=0$ and for all integers~$\gamma$, the estimate 
  \[
    \sum_{y\in\Ff_p^{\times}}
    U_{\alpha}\Bigl(\frac{K(y^2,1;p)}{\sqrt{p}}\Bigr)
    e\Bigl(\frac{\gamma y}{p}\Bigr)= O(p^{1/2})
  \]
  holds for all primes~$p$, where the implied constant depends only
  on~$\alpha$.
\end{lemma}

\begin{proof}
  This is a simple special case of the ``quasi-orthogonality''
  interpretation of Deligne's general form of the Riemann Hypothesis
  over finite fields. For instance, we can use the statement
  in~\cite{counting}*{Lemma\,3.5\,(2)} with the following data:
  \begin{itemize}
  \item the field~$k$ is~$\Ff_p$, the curve~$X$ is the
    multiplicative group and~$U=X$;
  \item the sheaf~$\mcF_1$ is the  $\alpha$-th
    symmetric power of the pullback by $y\mapsto y^2$ of the
    Kloosterman sheaf of rank~$2$, so that its trace function is given by
    \[
      t_{\mcF_1,\Ff_p}(y)=
      U_{\alpha}\Bigl(\frac{K(y^2,1;p)}{\sqrt{p}}\Bigr)
    \]
    for all $y\in U(\Ff_p)=\Ff_p^{\times}$;
  \item the sheaf~$\mcF_2$ is the restriction to $U$ of the Artin--Schreier
    sheaf $\mathcal{L}_{\psi_p(-\gamma y)}$, with
    \[
      \overline{t_{\mcF_2,\Ff_p}(y)}= e\Bigl(\frac{\gamma y}{p}\Bigr) \quad \text{for} \quad y\in\Ff_p^{\times}.
    \]
  \item the constant~$c$ is an upper-bound for the conductors of
    $\mcF_1$ and~$\mcF_2$, which can be taken to be a suitable
    constant depending only on~$\alpha$ (and not on~$p$) by formal
    properties of the conductor (although this can be done more
    elementarily in this case, it is maybe most convenient here to
    apply~\cite{qst}*{Prop.\,6.33} and~\cite{qst}*{Prop.\,7.5},
    respectively, taking into account~\cite{qst}*{Cor.\,7.4} to
    compare the notions of conductors from~\cite{counting}
    and~\cite{qst}).
  \end{itemize}

  Katz's study of the Kloosterman sheaf (cf.~\cite{gkm}*{Th.\,4.1.1})
  implies that~$\mcF_1$ is lisse, geometrically irreducible of
  rank~$\alpha+1$ and pure of weight~$0$; the sheaf~$\mcF_2$ is lisse,
  geometrically irreducible of rank~$1$ and pure of weight~$0$ (for
  essentially tautological reasons). Since~$\alpha+1\geq 2$, the two
  sheaves cannot be geometrically isomorphic, hence by loc. cit., we
  obtain the result (the key point being that the bound~$c$ for the
  conductors is independent of~$p$).
\end{proof}

\bibliographystyle{Nabbrv}
\bibliography{shortgraphs}

\end{document}